\documentclass[a4paper,12pt]{amsart}
\usepackage[utf8]{inputenc}
\usepackage[T1]{fontenc}

\usepackage{amssymb} % for extra symbols

\usepackage{hyperref}
\usepackage{mathtools}

\theoremstyle{plain}
\newtheorem{thm}{Theorem}[section]
\newtheorem{prop}[thm]{Proposition}

\newtheorem{lem}[thm]{Lemma}

\theoremstyle{definition}

\theoremstyle{remark}

\numberwithin{equation}{section}

\DeclareMathOperator{\Lip}{Lip}

\title[Lipschitz-free space over length space is LASQ but not ASQ]{The Lipschitz-free space over length space is locally almost square\\ but never almost square}
\author{Rainis Haller}%
\author{Jaan Kristjan Kaasik}%
\author{Andre Ostrak}%
\address{Institute of Mathematics and Statistics, University of Tartu, Narva mnt 18, 51009, Tartu, Estonia}%
\email{rainis.haller@ut.ee}%
\email{jaan.kristjan.kaasik@ut.ee}
\email{andre.ostrak@ut.ee}

\date{}

\begin{document}

\begin{abstract}
We prove that the Lipschitz-free space over a metric space $M$ is locally almost square whenever $M$ is a length space. Consequently, the Lipschitz-free space is locally almost square if and only if it has the Daugavet property. We also show that a Lipschitz-free space is never almost square.
\end{abstract}

\maketitle
\section{Introduction}
Throughout the paper we consider only nontrivial real Banach spaces and metric spaces which have more than one element.

Let $M$ be a metric space with a distinguished point $0\in M$. By $\Lip_0(M)$ we denote the Banach space of all Lipschitz functions from $M$ to $\mathbb R$ vanishing at $0$ with respect to the norm
\[
\|f\|=\sup\bigl\{\frac{|f(p)-f(q)|}{d(p,q)}\colon p,q\in M, p\neq q\bigr\},
\]
i.e., $\|f\|$ is the Lipschitz constant of $f\in \Lip_0(M)$. 
It is well known that $\operatorname{Lip}_0(M)$ is a dual space, whose canonical predual is the \emph{Lipschitz-free space}
\[
\mathcal F(M)\coloneqq \overline{\operatorname{span}}\{\delta_p\colon p\in M\}\subset \operatorname{Lip}_0(M)^\ast,
\]
where $\delta_p(f)=f(p)$ for every $f\in \operatorname{Lip}_0(M)$ (see, e.g, \cite{MR2030906, MR3792558} for the proof and for a more detailed study of these Banach spaces).

A metric space $(M, d)$ is called a \emph{length space} if, for every pair of different points $p, q \in M$, the distance $d(p, q)$ is equal to the infimum of the length of rectifiable simple paths joining them. In the definition of a length space, one usually does not specify that the paths are simple, i.e., injective. However, it seems to be well known and not hard to prove that this extra assumption can be made without loss of generality.

%A complete metric space $(M,d)$ is a length space if and only if it has approximate midpoints, i.e., for all $x, y \in M$ and $\varepsilon>0$ there exists $z\in M$ such that $\max{d(x, z), d(z, y)} \leq d(x, y)/2 +\varepsilon$.

We are interested in length spaces because of the following characterisation of certain geometric properties of the Lipschitz-free space $\mathcal F(M)$ and the Lipschitz functions space $\Lip_0(M)$.  

\begin{thm}\label{thm: M length iff F(M) DP}%[{\cite[Theorem~1.5]{MR3912824}}]
    Let $M$ be a complete metric space. Then the following assertions are equivalent:
    \begin{enumerate}
        \item $M$ is length;
        \item $\Lip_0(M)$, i.e., $\mathcal{F}(M)^*$, has the Daugavet property;
        \item $\mathcal F(M)$ has the Daugavet property;
        \item $\mathcal F(M)$ has the strong diameter $2$ property;
        \item $\mathcal F(M)$ has the diameter $2$ property;
        \item $\mathcal F(M)$ has the slice diameter $2$ property;
        \item the unit ball of $\mathcal F(M)$ does not have strongly exposed points;
        \item $M$ has property $(Z)$.
    \end{enumerate}
\end{thm}
This result is directly stated in \cite[Theorem~1.5]{MR3912824}. In fact, the main contribution of \cite{MR3912824} was the implication $(8)\Rightarrow (1)$. The implications $(3)\Rightarrow (1)$ and $(7)\Leftrightarrow (8)$ were already obtained in \cite{MR3794100}, and $(1)\Rightarrow (2)$ and $(1)\Rightarrow (8)$ were obtained in \cite{MR2379289}. The implications $(2)\Rightarrow (3)\Rightarrow (4)\Rightarrow(5)\Rightarrow(6)\Rightarrow (7)$ hold in the more general setting when $\mathcal{F}(M)$ is replaced with any Banach space.

According to \cite{MR3415738}, a Banach space $X$ with the unit sphere $S_X$ is said to be 
\begin{itemize}
    \item[(a)] \emph{almost square} (briefly, \emph{ASQ}) if, for every $\varepsilon>0$ and every finite subset $\{x_1,\dotsc,x_n\}$ of $S_X$, there exists $y\in S_X$ such that \[\|x_i\pm y\|\leq 1+\varepsilon\] holds for every $i\in\{1,\dotsc,n\}$;
    \item[(b)] \emph{locally almost square} (briefly, \emph{LASQ}) if, for every $\varepsilon>0$ and every $x\in S_X$, there exists $y\in S_X$ such that \[\|x\pm y\|\leq 1+\varepsilon.\]
\end{itemize}
These notions have been considered in several papers (e.g., in \cite{MR4073466, MR3466077, MR3619230}).
Since every ASQ Banach space has the strong diameter 2 property \cite[Proposition~2.5]{MR3415738} and every LASQ Banach space has the slice diameter 2 property \cite[Proposition~2.5]{MR3235311}, it is natural to wonder whether almost squareness or locally almost squareness of $\mathcal F(M)$ is also characterised by $M$ being length.

In this paper, we show that $\mathcal F(M)$ is never ASQ and that $\mathcal{F}(M)$ is LASQ whenever $M$ is length. 

However, some natural questions arise from these results.
By \cite{MR3415738, MR4023351}, an ASQ Banach space is weakly almost square (briefly, WASQ) and has the symmetric strong diameter $2$ property (briefly, SSD$2$P). Furthermore, if a Banach space is WASQ, then it is LASQ, and, if a Banach space has the SSD$2$P, then it has the strong diameter $2$ property. It is known that the Lipschitz-free space $\mathcal F([0,1])=L_1([0,1])$ is WASQ (see, e.g., \cite{MR3415738} or \cite{MR3235311}), but does not have the SSD$2$P (see, e.g., \cite{MR3917942}). We don't know whether the Lipschitz-free space $\mathcal{F}(M)$ is WASQ whenever $M$ is length or if there exists a Lipschitz-free space which has the SSD$2$P.

Every Lipschitz functions space with the Daugavet property has the SSD$2$P, but the converse is not true (see, e.g., \cite{MR3985517}). A very recent study \cite{HOP} shows that the diameter $2$ property, the strong diameter $2$ property, and the SSD$2$P are all different for the Lipschitz functions spaces. However, the metric characterisations of these diameter $2$ properties for the Lipschitz functions spaces seem to be unknown.

\subsection*{Notations} For $p,q\in M$ with $p\neq q$, 
we denote by $m_{p,q}$ the norm one element $\tfrac{\delta_p-\delta_q}{d(p,q)}$ in $\mathcal{F}(M)$. If $p=q$, then we consider $m_{p,q} = 0$. We call the elements $m_{p,q}$ the \emph{elementary molecules}.

Recall that a function $\gamma\colon[\alpha,\beta]\rightarrow M$ is called a \emph{path from $p$ to $q$}, when $\gamma$ is continuous and $\gamma(\alpha)=p$ and $\gamma(\beta)=q$. We denote a path from $p$ to $q$ usually by $\gamma_{p,q}$. For a rectifiable path $\gamma\colon [\alpha,\beta]\to M$ we denote its length by $L(\gamma)$, i.e.,
\[
L(\gamma) = \sup\bigl\{\sum_{i=0}^{n-1} d(\gamma(t_i), \gamma(t_{i+1})) \colon n\in \mathbb{N},\, \alpha=t_0\leq \dotsb\leq t_n=\beta \bigr\}.
\]
By default we assume that the domain interval for a path is $[0,1]$. 

Given a point $p$ in $M$ and $r>0$, we denote by $B(p,r)$ the open ball in $M$ centred at $p$ of radius $r$.

\section{Main lemma}
The following is our technical tool for obtaining that the Lipschitz-free space $\mathcal{F}(M)$ is LASQ whenever $M$ is length (see Theorem~\ref{thm: length is LASQ} below).
\begin{lem}\label{lem: main}
%Äkki lemmas tehagi läbi arcidega, saame tugevama tulemuse, mida mujal ka kasutada?
Let $M$ be a length space, let $n\in\mathbb{N}$, and set $C_n = (4^n-1)/{3}$. For every $\varepsilon>0$ and for every $x = \sum_{i=1}^n \lambda_i\, m_{p_i, q_i}\in \mathcal{F}(M)$ %Seda hulka pole defineeritud praegu
with $\lambda_1, \ldots, \lambda_n>0$ and $\sum_{i=1}^n \lambda_i\leq 1$, there exist $m\in \mathbb{N}$ with $m\leq C_n$, $\alpha_1, \ldots, \alpha_m>0$, $a_1,\ldots, a_m, b_1, \ldots, b_m\in M$, and simple paths $\gamma_{a_1,b_1},\ldots, \gamma_{a_m, b_m}$ such that
\begin{itemize}
    \item[(a)] $\|x - \sum_{j=1}^m \alpha_j\, m_{a_j, b_j}\|< \varepsilon$;
    \item[(b)] $\sum_{j=1}^m \alpha_j < \sum_{i=1}^n \lambda_i+\varepsilon$;
    \item[(c)] images of paths $\gamma_{a_1, b_1}, \ldots, \gamma_{a_m, b_m}$ are pairwise disjoint or may intersect only at endpoints;
    \item[(d)] for every $j\in \{1,\ldots, m\}$, we have $L(\gamma_{a_j, b_j}) < d(a_j, b_j)+ \varepsilon\delta_j$, where $\delta_j =\min\{1, \frac{d(a_j, b_j)}{C_n \alpha_j}\}$.
\end{itemize}
\end{lem}
\begin{proof}
The proof is by induction on $n$. The statement clearly holds for $n = 1$. Assume now that it holds for some particular $n$. Let $\varepsilon>0$ and let $x = \sum_{i=1}^{n+1}\lambda_i\, m_{p_i, q_i}\in \mathcal{F}(M)$ with $\lambda_1, \dotsc, \lambda_{n+1}>0$ and $\sum_{i=1}^{n+1} \lambda_i\leq 1$. We need only to consider the case where $p_{n+1}\neq q_{n+1}$. Let
\begin{align*}
    \varepsilon' = \frac{\varepsilon d(p_{n+1}, q_{n+1})}{C_{n+1}(1+d(p_{n+1}, q_{n+1}))}
\end{align*}
and let
\begin{align*}
    x' = \sum_{i=1}^n \lambda_i\, m_{p_i, q_i}.
\end{align*}
By assumption there exist $m\in \mathbb{N}$ with $m\leq C_n$, $\alpha_1', \ldots, \alpha_m'>0$, $a_1, \ldots, a_m, b_1,\ldots, b_m\in M$, and simple paths $\gamma_{a_1,b_1},\ldots, \gamma_{a_m, b_m}\colon [0,1]\to M$ such that
\begin{itemize}
    \item[(a$'$)] $\|x' - \sum_{j=1}^m \alpha_j'\, m_{a_j, b_j}\|< \varepsilon'$;
    \item[(b$'$)] $\sum_{j=1}^m \alpha_j' < \sum_{i=1}^n \lambda_i+\varepsilon'$;
    \item[(c$'$)] images of paths $\gamma_{a_1, b_1}, \ldots, \gamma_{a_m, b_m}$ are pairwise disjoint or may intersect only at endpoints;
    \item[(d$'$)] for every $j\in \{1,\ldots, m\}$, we have $L(\gamma_{a_j, b_j}) < d(a_j, b_j)+ \varepsilon'\delta_j$, where $\delta_j =\min\{1, \frac{d(a_j, b_j)}{C_n \alpha_j'}\}$.
\end{itemize}

Let
\[
y = \sum_{j = 1}^m \alpha_j'\, m_{a_j, b_j} + \lambda_{n+1}\, m_{p_{n+1}, q_{n+1}}.
\]
Then $\|x-y\| = \|x'- \sum_{j=1}^m \alpha_j'\, m_{a_j, b_j}\|< \varepsilon'< \varepsilon$. 

Since $M$ is a length space, there is a simple path $\gamma\colon[0,1]\to M$ from $p_{n+1}$ to $q_{n+1}$ such that $L(\gamma)< d(p_{n+1}, q_{n+1})+\varepsilon\delta$, where $\delta =\min\bigl\{1, \frac{d(p_{n+1},q_{n+1})}{C_{n+1}\lambda_{n+1}}\bigr\}$.
The path $\gamma$ may intersect the paths $\gamma_{a_1, b_1},\dotsc,$ $\gamma_{a_m, b_m}$ elsewhere than just at endpoints. Let us solve this problem.

We start by finding a special partition $0=t_0<\dotsb<t_l=1$ of the interval $[0,1]$ and by adjusting $\gamma$ in the subintervals of that partition if necessary. 
% By setting $u_0=\gamma(t_0),\dotsc,u_l=\gamma(t_l)$, and, for every $k\in \{1,\ldots, l\}$, 
% \[
% \beta_k = \lambda_{n+1}\frac{d(u_{k-1}, u_k)}{d(p_{n+1}, q_{n+1})},
% \]
% we will have 
% \[
% \lambda_{n+1}\, m_{p_{n+1}, q_{n+1}} = \sum_{k=1}^l \beta_k\, m_{u_{k-1}, u_k}.
% \]
Set $t_0=0$ and $u_0 = \gamma(t_0)$. We proceed by induction. Assume that we have found some $t_{k-1}<1$, where $k\in\mathbb N$. We will find $t_k$ and then set $u_k=\gamma(t_k)$ and
\[
\beta_k = \lambda_{n+1}\frac{d(u_{k-1}, u_k)}{d(p_{n+1}, q_{n+1})}.
\]
For any $j\in\{1,\dotsc,m\}$, let 
\[T_j=\{t\in[0,1]\colon \gamma(t)\in \operatorname{Im}\gamma_{a_j,b_j}\}.\]

If $\{t\in \bigcup_{j=1}^m T_j\colon t>t_{k-1}\}=\varnothing$, then set $t_{k}=1$ and $l=k$. Notice that $\gamma|_{(t_{k-1},t_k)}$ does not intersect any $\gamma_{a_j,b_j}$. From $u_{k-1}$ to $u_{k}$ we consider the path $\gamma_{u_{k-1},u_k}=\gamma|_{[t_{k-1},t_k]}$. By the choice of $\gamma$, we have
\begin{align*}
    L(\gamma_{u_{k-1},u_k}) &< d(u_{k-1}, u_k)+\varepsilon\delta %= d(u_{k-1}, u_k)+\varepsilon \min\Big\{1, \frac{d(p_{n+1}, q_{n+1})}{C_{n+1}\lambda_{n+1}}\Big\}
    =  d(u_{k-1}, u_k)+\varepsilon\min\bigl\{1,\frac{d(u_{k-1}, u_k)}{C_{n+1}\beta_k}\bigr\}.
\end{align*}

Otherwise, assume that $\{t\in \bigcup_{j=1}^m T_j\colon t>t_{k-1}\}\neq\varnothing$. Let \[J=\bigl\{j\in\{1,\dotsc,m\}\colon t_{k-1}\in T_j\bigr\}.\] If $\{t\in \bigcup_{j\in J}T_j\colon t>t_{k-1}\}=\varnothing$, then let 
\[t_{k}=\min\bigl\{t\in\bigcup_{\substack{j=1\\j\not\in J}}^m T_j\colon t>t_{k-1}\bigr\}.\] Again, notice that $\gamma|_{(t_{k-1},t_k)}$ does not intersect any $\gamma_{a_j,b_j}$. From $u_{k-1}$ to $u_{k}$ we consider the path $\gamma_{u_{k-1},u_k}=\gamma|_{[t_{k-1},t_k]}$. By the choice of $\gamma$, we have
\begin{align*}
    L(\gamma_{u_{k-1},u_k}) &< d(u_{k-1}, u_k)+\varepsilon\delta %= d(u_{k-1}, u_k)+\varepsilon \min\Big\{1, \frac{d(p_{n+1}, q_{n+1})}{C_{n+1}\lambda_{n+1}}\Big\}
    =  d(u_{k-1}, u_k)+\varepsilon\min\bigl\{1,\frac{d(u_{k-1}, u_k)}{C_{n+1}\beta_k}\bigr\}.
\end{align*}
But when $\{t\in \bigcup_{j\in J}T_j\colon t>t_{k-1}\}\neq\varnothing$, let 
\[
t_{k}=\max\bigl\{t\in \bigcup_{j\in J}T_j\colon t>t_{k-1}\bigr\}.
\]
In this case, there exists a $j\in J$ such that $u_k\in \operatorname{Im}\gamma_{a_j, b_j}$. Fix such a $j$. Let $r = \gamma_{a_j, b_j}^{-1}(u_{k-1})$ and $s = \gamma_{a_j, b_j}^{-1}(u_{k})$. If $r<s$, then define the path $\gamma_{u_{k-1}, u_k}$ as the restriction of the path $\gamma_{a_j, b_j}$ to $[r, s]$. Otherwise, define $\gamma_{u_{k-1}, u_k}$ as the restriction of the inverse path of $\gamma_{a_j, b_j}$ to $[1-r, 1-s]$.

We will now split $\gamma_{a_1, b_1}, \ldots, \gamma_{a_m, b_m}$. For every $j\in\{1,\dotsc,m\}$, let
\[S_j=\bigl\{t\in[0,1]\colon \gamma_{a_j,b_j}(t)\in\{\gamma(t_0),\ldots, \gamma(t_l)\}\bigr\},\]
and note that $|S_j|\leq 2$. If $S_j=\varnothing$, then let $c_j=a_j$ and $d_j=a_j$; other\-wise, let $c_j=\gamma_{a_j,b_j}(\min S_j)$ and $d_j=\gamma_{a_j,b_j}(\max S_j)$, and we consider the paths $\gamma_{a_j, c_j}$, $\gamma_{c_j, d_j}$, and $\gamma_{c_j, d_j}$ as the restrictions of path $\gamma_{a_j, b_j}$ to intervals $[0, \min S_j]$, $[\min S_j, \max S_j]$, and $[\max S_j, 1]$, respectively.
For every $j\in\{1,\dotsc,m\}$, let 
\[
\alpha_j=\alpha_j'\dfrac{d(a_j,c_j)}{d(a_j,b_j)},\qquad \alpha_{m+j}=\alpha_j'\dfrac{d(c_j,d_j)}{d(a_j,b_j)},\qquad \alpha_{2m+j}=\alpha_j'\dfrac{d(d_j,b_j)}{d(a_j,b_j)}.
\]

Note that 
\[
y=\sum_{j=1}^{m} \bigl(\alpha_j\, m_{a_j,c_j}+\alpha_{m+j}\, m_{c_j,d_j}+ \alpha_{2m+j}\, m_{d_j,b_j}\bigr)+\sum_{k=1}^l \beta_k\, m_{u_{k-1}, u_k}.
\]
We may and do assume that $a_j\neq c_j$, $c_j\neq d_j$, and $d_j\neq b_j$  for any particular $j\in\{1,\dotsc,m\}$; otherwise, the corresponding elementary molecule is $0$ and we drop it from this representation of $y$.

All the paths $\gamma_{a_j, c_j}$, $\gamma_{c_j, d_j}$, $\gamma_{d_j, b_j}$, and $\gamma_{u_{k-1}, u_k}$, where $j\in\{1,\dotsc,m\}$ and $k\in\{1,\dotsc,l\}$, are simple and pairwise disjoint or intersect just at the endpoints, except $\gamma_{c_j, d_j}$ and $\gamma_{u_{k-1}, u_k}$ may be the same paths or inverse paths to each other for some $j$ and $k$ (note that for every $k$ there exists at most one such $j$). If $\gamma_{c_j, d_j}$ and $\gamma_{u_{k-1}, u_k}$ are the same paths, then $m_{c_j,d_j} = m_{u_{k-1},u_k}$ and we gather these elementary molecules in the representation of $y$ by taking into account that \[\alpha_{m+j}\, m_{c_j,d_j}+\beta_k\, m_{u_{k-1},u_k} = (\alpha_{m+j}+\beta_k)\,m_{c_j,d_j}.\]
If $\gamma_{u_{k-1}, u_k}$ and $\gamma_{c_j, d_j}$ are inverse paths to each other, then $m_{u_{k-1},u_k}=-m_{c_j,d_j}$ and we gather these elementary molecules in the representation of $y$ by taking into account that \[\alpha_{m+j}\,m_{c_j,d_j}+\beta_k\, m_{u_{k-1},u_k}=(\alpha_{m+j}-\beta_k)\,m_{c_j,d_j}\]
if $\alpha_{m+j}\geq \beta_k$, and
\[\alpha_{m+j}\,m_{c_j,d_j}+\beta_k\, m_{u_{k-1},u_k} = (\beta_k-\alpha_{m+j})\,m_{d_j,c_j}\]
if $\alpha_{m+j}<\beta_k$. 

This representation of $y$ together with the remaining paths are the ones we were looking for. Note that the obtained representation of $y$ includes at most $4m+1$ elementary molecules. 
We have previously shown that $\|x-y\|<\varepsilon$. It is also clear that the images of all the corresponding paths are pairwise disjoint or intersect only at endpoints.

From $L(\gamma_{a_j, b_j}) < d(a_j, b_j)+ \varepsilon'\delta_j$, where $\delta_j =\min\{1, \frac{d(a_j, b_j)}{C_n \alpha_j'}\}$, it follows that 
\begin{align*}
L(\gamma_{a_j, c_j}) &< d(a_j, c_j)+ \varepsilon'\delta_j,\\
L(\gamma_{c_j, d_j}) &< d(c_j, d_j)+ \varepsilon'\delta_j,\\
L(\gamma_{d_j, b_j}) &< d(d_j, b_j)+ \varepsilon'\delta_j.
\end{align*}
We want to show that 
\begin{align*}
L(\gamma_{a_j, c_j}) &< d(a_j, c_j)+ \varepsilon\delta_j^1,\\
L(\gamma_{c_j, d_j}) &< d(c_j, d_j)+ \varepsilon\delta_j^2,\\
L(\gamma_{d_j, b_j}) &< d(d_j, b_j)+ \varepsilon\delta_j^3,
\end{align*}
where
\begin{align*}
\delta_j^1&=\min\bigl\{1,\frac{d(a_j, c_j)}{C_{n+1} \alpha_{j}}\bigr\},\\
\delta_j^2&=\min\bigl\{1, \frac{d(c_j, d_j)}{C_{n+1} (\alpha_{m+j}+\beta_k)}\bigr\},\\
\delta_j^3&=\min\bigl\{1, \frac{d(d_j, b_j)}{C_{n+1} \alpha_{2m+j}}\bigr\}.
\end{align*}
Notice that $\varepsilon>C_{n+1}\varepsilon'$, $C_{n+1}\delta_j^1\geq \delta_j$, and $C_{n+1}\delta_j^3\geq \delta_j$.
% Notice first that we clearly have $L(\gamma_{c_j, d_j}) < d(c_j, d_j)+ \varepsilon'\delta_j'$. This is because 
% \begin{align*}
%     L(\gamma_{c_j, d_j})&=L(\gamma_{a_j, b_j})-L(\gamma_{a_j, c_j})-L(\gamma_{d_j, b_j})\\
%     &<d(a_j,b_j)+\varepsilon'\delta_j'-d(a_j,c_j)-d(d_j,b_j)\\
%     &\leq d(c_j,d_j)+\varepsilon'\delta_j'.
% \end{align*}
Therefore it suffices to show that $\varepsilon'\delta_j\leq \varepsilon \delta_j^2$, or, equivalently,
\begin{align*}
\frac{\varepsilon d(p_{n+1},q_{n+1})}{C_{n+1}(1+d(p_{n+1},q_{n+1}))}&\min\bigl\{1,\dfrac{d(a_j,b_j)}{C_n\alpha_j'}\bigr\}\\\leq\varepsilon&\min\bigl\{1,\frac{d(c_j,d_j)}{C_{n+1}(\alpha_{m+j}+\beta_k)}\bigr\}.
\end{align*}
This inequality clearly holds if the latter minimum is $1$. We are left with proving that
\begin{align*}
\frac{d(p_{n+1},q_{n+1})}{1+d(p_{n+1},q_{n+1})} \min\bigl\{1,\dfrac{d(a_j,b_j)}{C_n\alpha_j'}\bigr\}&\leq \frac{d(c_j,d_j)}{\alpha_{m+j}+\beta_k}\\
&=\dfrac{1}{\frac{\alpha_j'}{d(a_j,b_j)}+\frac{\lambda_{n+1}}{d(p_{n+1},q_{n+1})}}.
\end{align*}
If $1\geq \frac{d(a_j,b_j)}{C_n\alpha_j'}$, then 
\begin{align*}
\dfrac{\alpha_j'}{d(a_j,b_j)}+\dfrac{\lambda_{n+1}}{d(p_{n+1},q_{n+1})}\leq \bigl(1+\dfrac{1}{d(p_{n+1},q_{n+1})}\bigr) \dfrac{C_n\alpha_j'}{d(a_j,b_j)}.
\end{align*}
If $1\leq \frac{d(a_j,b_j)}{C_n\alpha_j'}$, then 
\begin{align*}
\dfrac{\alpha_j'}{d(a_j,b_j)}+\dfrac{\lambda_{n+1}}{d(p_{n+1},q_{n+1})}\leq \dfrac{1}{C_n}+\dfrac{1}{d(p_{n+1},q_{n+1})}\leq 1+\dfrac{1}{d(p_{n+1},q_{n+1})}.
\end{align*}

It remains to show that
\begin{align*}
    \sum_{j=1}^{3m} \alpha_j+\sum_{k=1}^l\beta_k\leq \sum_{i=1}^{n+1}\lambda_i +\varepsilon.
\end{align*}
First, notice that
\begin{align*}
\sum_{j=1}^{3m} \alpha_j&=\sum_{j=1}^m \alpha_j'\bigl(\frac{d(a_j,c_j)}{d(a_j,b_j)}+\frac{d(c_j,d_j)}{d(a_j,b_j)}+\frac{d(d_j,b_j)}{d(a_j,b_j)}\bigr)\\
&=\sum_{j=1}^m \alpha_j'\frac{d(a_j,c_j)+d(c_j,d_j)+d(d_j,b_j)}{d(a_j,b_j)}\\
&\leq \sum_{j=1}^m \alpha_j'\frac{L(\gamma_{a_j,b_j})}{d(a_j,b_j)}< \sum_{j=1}^m \alpha_j'\frac{d(a_j,b_j)+\varepsilon'\delta_j}{d(a_j,b_j)}\\
&=\sum_{j=1}^m \alpha_j'+\varepsilon'\sum_{j=1}^m \frac{\alpha_j'\delta_j}{d(a_j,b_j)}\\
&\leq \sum_{j=1}^m\alpha_j' + \varepsilon' m
< \sum_{i=1}^n\lambda_i + (C_n+1)\varepsilon'\\
&< \sum_{i=1}^n\lambda_i + \frac{C_n+1}{C_{n+1}}\varepsilon
< \sum_{i=1}^n\lambda_i + \frac{\varepsilon}{2}
\end{align*}
and
\begin{align*}
\sum_{k=1}^l \beta_k&=\lambda_{n+1}\sum_{k=1}^l \frac{d(\gamma(t_{k-1}),\gamma(t_{k}))}{d(p_{n+1},q_{n+1})}\leq\frac{\lambda_{n+1}L(\gamma)}{d(p_{n+1},q_{n+1})}\\
&<\lambda_{n+1}+\frac{\lambda_{n+1}\varepsilon\delta}{d(p_{n+1},q_{n+1})}\leq \lambda_{n+1}+\frac{\varepsilon}{2}.
\end{align*}
Therefore
\begin{align*}
    \sum_{j=1}^{3m} \alpha_j+\sum_{k=1}^l\beta_k< \sum_{i=1}^{n+1}\lambda_i+\varepsilon.
\end{align*}
\end{proof}

\section{Locally almost square Lipschitz-free spaces}
The following theorem characterises locally almost square Lipschitz-free spaces.
\begin{thm}\label{thm: length is LASQ}
Let $M$ be a complete metric space. The Lipschitz-free space $\mathcal F(M)$ is LASQ if and only if $M$ is length.
\end{thm}
\begin{proof} If $\mathcal F(M)$ is LASQ, then $\mathcal F(M)$ has the slice diameter $2$ property, and therefore $M$ is length by Theorem~\ref{thm: M length iff F(M) DP}.

Assume now that $M$ is a length space. Let $x\in S_{\mathcal F(M)}$ and let $\varepsilon>0$. By Lemma~\ref{lem: main}, there exist $m\in\mathbb N$, $\alpha_1, \ldots, \alpha_m>0$, $a_1,\ldots, a_m, b_1, \ldots, b_m\in M$, and simple paths $\gamma_{a_1,b_1},\ldots, \gamma_{a_m, b_m}$ such that
\begin{itemize}
    \item[(a)] $\|x - \sum_{j=1}^m \alpha_j\, m_{a_j, b_j}\|< \varepsilon/5$;
    \item[(b)] $\sum_{j=1}^m \alpha_j < 1+\varepsilon/5$;
    \item[(c)] images of paths $\gamma_{a_1, b_1}, \ldots, \gamma_{a_m, b_m}$ are pairwise disjoint or intersect only at endpoints;
    \item[(d)] for every $j\in \{1,\ldots, m\}$, we have $L(\gamma_{a_j, b_j}) < d(a_j, b_j)+ \varepsilon\delta_j/5$, where $\delta_j =\frac{d(a_j, b_j)}{m \alpha_j}$.
\end{itemize}
It suffices to find $y_1,\dotsc,y_m\in\mathcal F(M)$ and a Lipschitz function $g\in S_{\Lip_0(M)}$ (we fix $0\in M$ below) such that, for every $j\in \{1,\dotsc, m\}$, one has $g(y_j)=1$ and
\[
\|m_{a_j, b_j}\pm y_j\| \leq 1+\frac{\varepsilon}{5m\alpha_j}.
\]
Indeed, assume that such $y_1,\dotsc, y_m$ and $g$ exist and let $y = \sum_{j=1}^m \alpha_j y_j$. Then
\[
\|y\|\geq g\bigl(\sum_{j=1}^m \alpha_j y_j\bigr)=\sum_{j=1}^m \alpha_j g(y_j)=\sum_{j=1}^m \alpha_j> 1-\frac{\varepsilon}{5}
\]
and
\begin{align*}
\|y\| & \leq \sum_{j=1}^m\alpha_j \|{y_j}\|=
\sum_{j=1}^m\frac{\alpha_j}{2}\|m_{a_j,b_j}+y_j-m_{a_j,b_j}+y_j\|\\&\leq \sum_{j=1}^m\frac{\alpha_j}{2}\bigl(\|m_{a_j,b_j}+y_j\|+\|m_{a_j,b_j}-y_j\|\bigr)\\&\leq \sum_{j=1}^m\alpha_j\bigl(1+\frac{\varepsilon}{5m\alpha_j}\bigr)< 1+\frac{2\varepsilon}{5},
\end{align*}
and therefore
\begin{align*}
\bigl\|x\pm\frac{y}{\|y\|}\bigr\|&\leq \bigl\|x-\sum_{j=1}^m\alpha_j\, m_{a_j,b_j}\bigr\|+\bigl\|\sum_{j=1}^m\alpha_j\, m_{a_j,b_j}\pm y\bigr\|+\bigl|\|y\|-1 \bigr|\\
&< \frac{\varepsilon}{5}+\sum_{j=1}^m\alpha_j \|m_{a_j,b_j}\pm y_j\|+\frac{2\varepsilon}{5}\leq 1+\varepsilon.
\end{align*}

For every $j\in\{1,\dotsc, m\}$, we will define $K\in\mathbb N$ and a finite subset $\{u_0,\dotsc,u_{2K}\}$ of $\operatorname{Im}\gamma_{a_j,b_j}$, then we define $y_j$ as a linear combination $\sum_{k=1}^{2K} \lambda_k \,m_{u_k,u_{k-1}}$ with specific $\lambda_1,\dotsc,\lambda_{2K}\in \mathbb{R}$, and finally, we define $g$ on $\{u_0,\dotsc,u_{2K}\}$. 

Let 
\[A=\{a_1,\dotsc,a_m,b_1,\dotsc,b_m\}.\]
Choose $R>0$ such that
\begin{enumerate}
    \item for all $p,q\in A$ with $p\neq q$, one has $d(p,q)>2R$;
    \item for every $p\in A$ and every $j\in\{1,\dotsc,m\}$ with $p\neq a_j$ and $p\neq b_j$, one has
$B(p,R)\cap \operatorname{Im}\gamma_{a_j,b_j}=\varnothing$,
\end{enumerate}
and let $B=\bigcup_{p\in A}B(p,R)$.

The sets $\operatorname{Im} \gamma_{a_1,b_1}{\setminus} B, \dotsc, \operatorname{Im} \gamma_{a_m,b_m}{\setminus} B$ are pairwise disjoint and compact, therefore there exists $r\in(0,R)$ such that, for all $i,j\in \{1,\dotsc, m\}$ with $i\neq j$, one has
\[
d(\operatorname{Im} \gamma_{a_i,b_i}{\setminus} B , \operatorname{Im} \gamma_{a_j,b_j}{\setminus} B) >2r.
\]

We start by defining a finite subset of $\operatorname{Im} \gamma_{a_1,b_1}$.
Let $K\in\mathbb N$ be such that
\[
2(K-2)r> d(a_1,b_1)-2R
\]
and let
\[
s = \frac{d(a_1,b_1)-2R}{2(K-2)}.
\]
Then $0<s<r$.

We will find a special partition of the interval $[0,1]$ and the corresponding points in $\operatorname{Im} \gamma_{a_1,b_1}$. 
Let $t_0=0$ and $t_{2K}=1$, and let $u_0 = a_1$ and $u_{2K}=b_1$. Let 
\[
t_1=\max\bigl\{t\colon d(\gamma_{a_1,b_1}(t),a_1)=R\bigr\}
\]
and
\[
t_{2K-1} = \min\bigl\{t \colon d(\gamma_{a_1, b_1}(t), b_1)=R\bigr\},
\]
and let $u_1=\gamma_{a_1, b_1}(t_1)$ and $u_{2K-1}=\gamma_{a_1, b_1}(t_{2K-1})$. Assume that we have found some $t_k$ and $u_k$, where $1\leq k\leq 2K-4$. Let
\[
t_{k+1} = \max\bigl\{t \colon d(\gamma_{a_1,b_1}(t),u_{k})=s\bigr\}
\]
and let $u_{k+1} = \gamma_{a_1, b_1}(t_{k+1})$. 

Note that $t_{2K-3}\leq t_{2K-1}$. Let $t_{2K-2}$ be such that $t_{2K-3}\leq t_{2K-2}\leq t_{2K-1}$ and 
\[
d\bigl(u_{2K-3}, \gamma_{a_1,b_1}(t_{2K-2})\bigr)=d\bigl(u_{2K-1}, \gamma_{a_1,b_1}(t_{2K-2})\bigr),
\]
and let $u_{2K-2}=\gamma_{a_1,b_1}(t_{2K-2})$. Then $0 = t_0\leq\dotsc \leq t_{2K} =1$ and the obtained set $\{u_0,\dotsc, u_{2K}\}$ is a finite subset of $\operatorname{Im}\gamma_{a_1,b_1}$, where $u_k = \gamma_{a_1,b_1}(t_k)$. In order to make $M$ into a pointed metric space, set $0=u_1$.

Define
\begin{align*}
y_1&=\sum_{k=1}^{K} \bigl(\frac{d(u_{2k-2},u_{2k-1})}{d(a_1,b_1)}\,m_{u_{2k-2},u_{2k-1}}-\frac{d(u_{2k-1},u_{2k})}{d(a_1,b_1)}\,m_{u_{2k-1},u_{2k}}\bigr).
\end{align*}
Let us show that
\[
\|m_{a_1,b_1}+ y_1\|\leq 1+\frac{\varepsilon}{{5m\alpha_1}}.
\]
For every Lipschitz function $f\in S_{\Lip_0(M)}$, we have
\begin{align*}
&d(a_1,b_1)\bigl|f( m_{a_1,b_1} + y_1)\bigr|\\
&\phantom{\;\;\;}=\bigl|f(u_0)-f(u_{2K}) + \sum_{k=1}^{K} \bigl(f(u_{2k-2})-f(u_{2k-1}) -f(u_{2k-1})+f(u_{2k}) \bigr)\bigr|\\
&\phantom{\;\;\;}= 2\,\bigl|\sum_{k=1}^{K} \bigl(f(u_{2k-2})-f(u_{2k-1}) \bigr)\bigr|
\leq 2\sum_{k=1}^{K}d(u_{2k-2},u_{2k-1})\\
&\phantom{\;\;\;}=\sum_{k=1}^{2K}d(u_{k-1}, u_{k})
\leq L(\gamma_{a_1,b_1})\leq d(a_1,b_1)+\frac{\varepsilon\delta_1}{5}.
\end{align*}
Therefore
\[
\| m_{a_1,b_1} + y_1\|\leq 1+\frac{\varepsilon}{5m\alpha_1}.
\]
Analogously one can show that
\[
\|m_{a_1,b_1}- y_1\|\leq 1+\frac{\varepsilon}{{5m\alpha_1}}.
\]
Define the function $g$ on $\{u_0,\dotsc,u_{2K}\}$ by
\[
g(p)=\begin{cases}
R&\text{ if $p=u_0$ or $p=u_{2K}$};\\
0&\text{ if $p=u_k$, where $k$ is odd or $p=u_{2K-2}$;}\\
s&\text{ in other cases.}
\end{cases}
\]
Note that the Lipschitz constant of $g$ on the set $\{u_0,\dotsc, u_{2K}\}$ is equal to $1$ and that
\begin{align*}
    g(y_1)&=\frac{1}{d(a_1,b_1)}\sum_{k=1}^{K} \big(g(u_{2k-2})-g(u_{2k-1})-g(u_{2k-1})+g(u_{2k})\big)\\
        &=\frac{1}{d(a_1,b_1)}(2R+2(K-2)s)=1.
\end{align*}

For every $i\in \{2,\dotsc, m\}$, we analogously define a finite subset of $\operatorname{Im}\gamma_{a_i,b_i}$, an $y_i\in \mathcal{F}(M)$, and extend the definition of $g$ onto the finite subset. Let $C$ be the set of all points where $g$ is defined after this process. The function $g$ is correctly defined on the set $C$ because the images of paths $\gamma_{a_1,b_1},\dotsc,\gamma_{a_m,b_m}$ are pairwise disjoint or intersect only at the endpoints, where the values of $g$ are equal to $R$. By McShane's extension theorem, extend $g$ to the whole of $M$ while preserving the Lipschitz constant. 

It remains to verify that the Lipschitz constant of $g|_C$ is equal to $1$. Let $p,q\in C$ with $p\neq q$. If $p\in B\cap C$ or $q\in B\cap C$, then $d(p,q)\geq R$ because $B\cap C = A$, and since $0\leq g(p)\leq R$ and $0\leq g(q)\leq R$, it follows that
\[
|g(p)-g(q)|\leq R\leq d(p,q). 
\]

Assume now that $p,q\in C{\setminus} B$. Then there are uniquely determined $i,j\in \{1,\dotsc, m\}$ such that $p\in \operatorname{Im}\gamma_{a_i, b_i}{\setminus} B$ and $q\in \operatorname{Im}\gamma_{a_j, b_j}{\setminus} B$. If $i = j$, then we are done. Otherwise $d(p,q)>2r$ by our choice of $r$, and since $0\leq g(p)< r$ and $0\leq g(q)< r$, it follows that
\[
|g(p)-g(q)|< r < d(p,q).
\]
\end{proof}

\section{Lipschitz-free space is never almost square}

According to \cite{MR4183387}, a Banach space $X$ is \emph{$s$-almost square} (briefly, \emph{s-ASQ}) for $s\in (0, 1]$ 
if, for every $\varepsilon>0$ and every finite subset $\{x_1,\dotsc,x_n\}$ of $S_X$, there exists $y\in S_X$ such that \[\|x_i\pm sy\|\leq 1+\varepsilon\] holds for every $i\in\{1,\dotsc,n\}$.
Note that 1-ASQ means precisely ASQ. By \cite[Propositions~1.6 and 1.7]{MR4183387}, every slice of the unit ball of an $s$-ASQ Banach space has diameter equal to $2s$, and therefore the unit ball of an $s$-ASQ Banach space cannot contain strongly exposed points. Thus, by Theorem~\ref{thm: M length iff F(M) DP}, a complete metric space $M$ has to be length whenever the Lipschitz-free space $\mathcal{F}(M)$ is $s$-ASQ.

Our main result in this section is the following.

\begin{thm}\label{thm: F(M) never ASQ}
Lipschitz-free space is never ASQ. In fact, Lipschitz-free space is not $s$-ASQ for any $0<s\leq 1$.
\end{thm}

This theorem is a straightforward consequence of the following proposition.

\begin{prop}\label{prop: F(M) never 2-ASQ}
Let $M$ be a length space. For every $\varepsilon>0$, there exist $n\in\mathbb N$ and $p_1,q_1,\dotsc, p_n, q_n\in M$ such that, for every $y \in B_{\mathcal{F}(M)}$, 
\begin{equation}\label{eq: s-asq}
\max_{i\in \{1,\dotsc,n\}}\|m_{p_i, q_i}+y\|>1+\|y\|-\varepsilon.    
\end{equation}
\end{prop}
\begin{proof}
Let $\varepsilon>0$. Let $n\in \mathbb{N}$ be such that $8/n\leq \varepsilon$ and let $\theta>0$ be such that $8\theta\leq \varepsilon$. Let $p,q\in M$ with $p\neq q$,  let $r = d(p,q)/(2n)$, and let $\gamma$ be a rectifiable path from $p$ to $q$. We define points $p_1,\dotsc, p_n\in M$ inductively. Set $p_1=p$ and, if we have fixed $p_i$ for some $i\in \{1,\dotsc, n-1\}$, then we define $p_{i+1} = \gamma(\max A_i)$, where
\[
A_i = \{t\in [0,1] \colon d(\gamma(t), p_i) = d(p,q)/n\}.
\]
Let $q_1,\dotsc, q_n \in M$ be such that $q_i \in B(p_i, \theta r){\setminus}\{p_i\}$ for every $i\in\{1,\dotsc,n\}$. 

Let $y\in B_{\mathcal{F}(M)}$ and then let $f\in S_{\Lip_0(M)}$ be such that $f(y)=\|y\|$. In order to show that the inequality (\ref{eq: s-asq}) holds, it suffices to consider only the case $y = \sum_{j=1}^m \frac{1}{m}\, m_{u_j, v_j}$.
For every $i\in\{1,\dotsc,n\}$, set $B_i = B(p_i, r)$. Then $B_1,\dotsc,B_n$ are pairwise disjoint. 

Fix $i\in \{1,\dotsc, n\}$ such that the set
\[
J = \{j\in \{1,\dotsc,m\}\colon u_j\in B_i \text{ or } v_j\in B_i\}
\]
has at most $2m/n$ elements. Define $g\in \Lip_0(M)$ by setting $g|_{M{\setminus} B_i}=f|_{M{\setminus} B_i}$, $g(q_i)=f(q_i)$, $g(p_i)=f(q_i)+d(p_i,q_i)$, and, by McShane's extension theorem, extending $g$ to be defined in the whole of $M$ while preserving the Lipschitz constant.
Then $\|g\|\leq 1+2\theta$ because, for every $a\not\in B_i$, one has
\begin{align*}
    |g(a)-g(p_i)|&=|f(a)-f(q_i)+d(p_i,q_i)|\leq d(a,q_i)+d(p_i,q_i)\\
    &\leq d(a,p_i)+2d(p_i,q_i)\leq  d(a,p_i)+2\theta r\leq (1+2\theta)d(a,p_i).
\end{align*}
Note that
\begin{align*}
(f-g)(y) &= (f-g)\bigl(\sum_{j\in J}\frac{1}{m}\, m_{u_j,v_j}\bigr) \leq (\|f\|+\|g\|) \sum_{j\in J}\frac{1}{m}\|m_{u_j, v_j}\|\\
&\leq (2+2\theta)\frac{2m}{mn} = \frac{4(1+\theta)}{n}.
\end{align*}
Therefore
\begin{align*}
    g(y) = f(y)-(f-g)(y)\geq \|y\|-\frac{4(1+\theta)}{n},
\end{align*}
and hence
\begin{align*}
    \|m_{p_i,q_i}+y\|&\geq \frac{g}{\|g\|}(m_{p_i,q_i}+y)
    \geq \frac{1}{1+2\theta} \bigl(\frac{g(p_i)-g(q_i)}{d(p_i,q_i)}+g(y)\bigr)\\&\geq \frac{1}{1+2\theta} \bigl(1+\|y\|-\frac{4(1+\theta)}{n}\bigr)\\
    &> 1+\|y\|-2\theta(1+\|y\|)-\frac{4}{n}\geq 1+\|y\|-\varepsilon.
\end{align*}
\end{proof}

\section*{Acknowledgements}
This work was supported by the Estonian Research
Council grant (PRG1901).
The research of A.~Ostrak was supported by the University of Tartu ASTRA Project PER ASPERA, financed by the European Regional Development Fund.

\bibliographystyle{amsplain}
\bibliography{references}
%\printbibliography

\end{document}